\documentclass[10pt]{amsart}
\usepackage{amscd, amsmath, amsthm, amssymb}

\newcommand{\bbC}{\mathbb{C}}

\newcommand{\bbP}{\mathbb{P}}

\newcommand{\bbZ}{\mathbb{Z}}

\newcommand\Pic{{\text{Pic}}}

\newtheorem{theorem}{Theorem}[section]
\newtheorem{lemma}[theorem]{Lemma}
\newtheorem{proposition}[theorem]{Proposition}

\theoremstyle{definition}     

\theoremstyle{remark}
\newtheorem{remark}[theorem]{Remark}
\numberwithin{equation}{section}

\DeclareMathOperator{\rank}{rank}
\begin{document}
\title[Toward a geometric construction of fake projective planes]
{Toward a geometric construction of fake projective planes}

\author[J. Keum]{JongHae Keum }
\address{School of Mathematics, Korea Institute for Advanced
Study, Seoul 130-722, Korea } \email{jhkeum@kias.re.kr}
\thanks{Research supported by Basic Science Research Program through the National Research Foundation(NRF)
of Korea funded by the Ministry of Education, Science and
Technology (2007-C00002)} 
\begin{abstract} We give a criterion for a projective surface to become a quotient
of a fake projective plane. We also give a detailed information on
the elliptic fibration of a $(2,3)$-elliptic surface that is the
minimal resolution of a quotient of a fake projective plane. As a
consequence, we give a classification of ${\mathbb Q}$-homology
projective planes with cusps only.
\end{abstract}
\maketitle
It is known that a compact complex surface with the same Betti numbers as the
complex projective plane $\bbP^2$ is projective (see e.g. \cite{BHPV}).
Such a surface is called {\it a fake projective plane} if it is not
isomorphic to $\bbP^2$.

Let $X$ be a fake projective plane. Then its canonical bundle is
ample. So a fake projective plane is exactly a surface of general
type with $p_g(X)=0$ and $c_1(X)^2=3c_2(X)=9$. By \cite{Aubin} and
\cite{Yau}, its universal cover is the unit 2-ball ${\bf
B}\subset\bbC^2$ and hence its fundamental group $\pi_1(X)$ is
infinite. More precisely, $\pi_1(X)$ is exactly a discrete
torsion-free cocompact subgroup $\Pi$ of $PU(2,1)$ having minimal
Betti numbers and finite abelianization. By Mostow's
 rigidity theorem \cite{Mos}, such a ball quotient is strongly
rigid, i.e., $\Pi$ determines a fake projective plane up to
holomorphic or anti-holomorphic isomorphism. By \cite{KK}, no fake
projective plane can be anti-holomorphic to itself. Thus the
moduli space of fake projective planes consists of a finite number
of points, and the number is the double of the number of distinct
fundamental groups $\Pi$. By Hirzebruch's proportionality
principle \cite{Hir},
  $\Pi$ has covolume 1 in $PU(2,1)$. Furthermore, Klingler
  \cite{Kl} proved that the discrete torsion-free
cocompact subgroups of $PU(2,1)$ having minimal Betti numbers are
arithmetic (see also \cite{Ye}).

With these information, Prasad and Yeung \cite{PY} carried out a
classification of fundamental groups of fake projective planes.
They describe the algebraic group $\bar{G}(k)$ containing a
discrete torsion-free cocompact arithmetic subgroup $\Pi$ having
minimal Betti numbers and finite abelianization as follows. There
is a pair $(k,l)$ of number fields, $k$ is totally real, $l$ a
totally complex quadratic extension of $k$. There is a central
simple algebra $D$ of degree 3 with center $l$ and an involution
$\iota$ of the second kind on $D$ such that $k=l^{\iota}$. The
algebraic group $\bar{G}$ is defined over $k$ such that
$$\bar{G}(k)\cong\{z\in D| \iota(z)z=1\}/\{t\in
 l|\iota(t)t=1\}.$$
 There is one Archimedean place $\nu_0$ of $k$ so that
 $\bar{G}(k_{\nu_0})\cong PU(2,1)$ and $\bar{G}(k_{\nu})$ is
 compact for all other Archimedean places $\nu$.
 The data $(k,l, D, \nu_0)$ determines $\bar{G}$ up to $k$-isomorphism.
Using Prasad's volume formula \cite{P}, they were able to
eliminate most $(k,l, D,
 \nu_0)$, making a short list of possibilities where $\Pi$'s might
 occur, which yields a short list of maximal arithmetic subgroups $\bar{\Gamma}$ which might contain a
 $\Pi$. If $\Pi$ is contained, up to conjugacy, in a unique $\bar{\Gamma}$,
 then the group $\Pi$ or the fake projective plane ${\bf
B}/\Pi$ is said to belong to the {\it class} corresponding to the
conjugacy class of $\bar{\Gamma}$. If $\Pi$ is contained in two
non-conjugate maximal arithmetic subgroups, then $\Pi$  or ${\bf
B}/\Pi$ is said to form a {\it class} of its own.
 They exhibited 28 non-empty classes (\cite{PY},
Addendum). It turns out that the index of such a $\Pi$ in a
$\bar{\Gamma}$ is 1, 3, 9, or 21, and all such $\Pi$'s in the same
$\bar{\Gamma}$ have the same index.

Then Cartwright and Steger \cite{CS} have carried out a
computer-based but very complicated group-theoretic computation,
showing that there are exactly 28 non-empty classes, where 25 of
them correspond to conjugacy classes of maximal arithmetic
subgroups and each of the remaining 3 to a $\Pi$ contained in two
non-conjugate maximal arithmetic subgroups. This yields a complete
list of fundamental groups of fake projective planes: the moduli
space consists of exactly 100 points, corresponding to 50 pairs of
complex conjugate fake projective planes.

It is easy to see that the automorphism group $Aut(X)$ of a fake
projective plane $X$ can be given by
$$Aut(X)\cong N(\pi_1(X))/\pi_1(X),$$ where $N(\pi_1(X))$ is the
normalizer of $\pi_1(X)$ in a suitable $\bar{\Gamma}$.

\begin{theorem}\cite{PY},\cite{CS},\cite{CS2} For a fake
projective plane $X$,
$$Aut(X)=\{1\},\,\, C_3,\,\,
C_3^2,\,\, 7:3,$$ where $C_n$ denotes the cyclic group of order
$n$, and $7:3$ the unique non-abelian group of order $21$.
\end{theorem}

 According to (\cite{CS},\cite{CS2}), 68 of the 100 fake projective planes admit a nontrivial group of
automorphisms.

Let $(X, G)$ be a pair of a fake projective plane $X$ and a
non-trivial group $G$ of automorphisms. In \cite{K08}, all
possible structures of the quotient surface $X/G$ and its minimal
resolution were classified:

\begin{theorem}\cite{K08}\label{k08}
\begin{enumerate}
\item If $G=C_3$, then $X/G$ is a ${\mathbb Q}$-homology
projective plane with $3$ singular points of type
$\frac{1}{3}(1,2)$ and its minimal resolution is a minimal surface
of general type with $p_g=0$ and $K^2=3$. \item If $G=C_3^2$, then
$X/G$ is a ${\mathbb Q}$-homology projective plane with $4$
singular points of type $\frac{1}{3}(1,2)$ and its minimal
resolution is a minimal surface of general type with $p_g=0$ and
$K^2=1$. \item If $G=C_7$, then $X/G$ is a ${\mathbb Q}$-homology
projective plane with $3$ singular points of type
$\frac{1}{7}(1,5)$ and its minimal resolution is a $(2,3)$-,
$(2,4)$-, or $(3,3)$-elliptic surface. \item If $G=7:3$, then
$X/G$ is a ${\mathbb Q}$-homology projective plane with $4$
singular points, 3 of type $\frac{1}{3}(1,2)$ and one of type
$\frac{1}{7}(1,5)$, and its minimal resolution is a $(2,3)$-,
$(2,4)$-, or $(3,3)$-elliptic surface.
\end{enumerate}
\end{theorem}

Here a ${\mathbb Q}$-homology projective plane is a normal
projective surface with the same Betti numbers as $\bbP^2$. A fake
projective plane is a nonsingular ${\mathbb Q}$-homology
projective plane, hence every quotient is again a ${\mathbb
Q}$-homology projective plane. An $(a,b)$-elliptic surface is a
relatively minimal elliptic surface over $\bbP^1$ with two
multiple fibres of multiplicity $a$ and $b$ respectively. It has
Kodaira dimension 1 if and only if $a\ge 2, b\ge 2, a+b\ge 5$. It
is an Enriques surface iff $a=b=2$, and it is rational iff $a=1$
or $b=1$. An $(a,b)$-elliptic surface has $p_g=q=0$, and by
\cite{D} its fundamental group is the cyclic group of order the
greatest common divisor of $a$ and $b$.

\begin{remark} (1) Since $X/G$ has rational singularities only, $X/G$
and its minimal resolution have the same fundamental group. Let
$\bar{\Gamma}$ be the maximal arithmetic subgroup of $PU(2,1)$
containing $\pi_1(X)$. There is a subgroup $\tilde{G}\subset
\bar{\Gamma}$ such that $\pi_1(X)$ is normal in $\tilde{G}$ and
$G=\tilde{G}/\pi_1(X).$  Thus, $$X/G\cong {\bf B}/\tilde{G}.$$ It
is well known (cf. \cite{Arm}) that
$$\pi_1({\bf B}/\tilde{G})\cong \tilde{G}/H,$$
where $H$ is the minimal normal subgroup of $\tilde{G}$ containing
all elements acting non-freely on the 2-ball $\bf B$. In our
situation, it can be shown that $H$ is generated by torsion
elements of $\tilde{G}$, and Cartwright and Steger
 have computed, along with their computation of the
fundamental groups, the quotient group $\tilde{G}/H$ for each pair
$(X, G)$.
\begin{itemize}
\item \cite{CS} If $G=C_3$, then $$\pi_1(X/G)\cong \{1\},\, C_2,\, C_3,\, C_4,\, C_6,\, C_7,\, C_{13},\, C_{14},\,
C_2^2,\, C_2\times C_4,\, S_3,\, D_8\,\,{\rm or}\,\, Q_8,$$ where
$S_3$ is the symmetric group of order 6, and $D_8$ and $Q_8$ are
the dihedral and quaternion groups of order 8. \item \cite{CS2} If
$G=C_3^2$ or $C_7$ or $7:3$, then
$$\pi_1(X/G)\cong \{1\}\,\,{\rm or}\,\, C_2.$$
 This eliminates the possibility of $(3,3)$-elliptic surfaces in
Theorem \ref{k08}, as $(3,3)$-elliptic surfaces have $\pi_1=C_3$.
\end{itemize}

\medskip\noindent
(2) It is interesting to consider all ball quotients which are
covered irregularly by a fake projective plane. Indeed, Cartwright
and Steger have considered all subgroups $\tilde{G}\subset
PU(2,1)$ such that $\pi_1(X)\subset \tilde{G}\subset\bar{\Gamma}$
for some maximal arithmetic subgroup $\bar{\Gamma}$ and some fake
projective plane $X$, where $\pi_1(X)$ is not necessarily normal
in $\tilde{G}$. It turns out \cite{CS2} that, if $\pi_1(X)$ is not
normal in $\tilde{G}$, then there is another fake projective plane
$X'$ such that $\pi_1(X')$ is normal in $\tilde{G}$, hence ${\bf
B}/\tilde{G}\cong X'/G'$ where $G'=\tilde{G}/\pi_1(X')$. Thus such
a general subgroup $\tilde{G}$ does not produce a new surface.
\end{remark}

\bigskip
 It is a major step
toward a geometric construction of a fake projective plane to
construct a ${\mathbb Q}$-homology projective plane satisfying one
of the descriptions (1)-(4) from Theorem \ref{k08}. Suppose that
one has such a ${\mathbb Q}$-homology projective plane. Then, can
one construct a fake projective plane by taking a suitable cover?
In other words, does the description (1)-(4) from Theorem
\ref{k08} characterize the quotients of fake projective planes?
The answer is affirmative in all cases.

\begin{theorem}\label{main1} Let $Z$ be a ${\mathbb Q}$-homology
projective plane satisfying one of the descriptions $(1)$-$(4)$
from Theorem \ref{k08}.
\begin{enumerate}
\item If $Z$ is a ${\mathbb Q}$-homology
projective plane with $3$ singular points of type
$\frac{1}{3}(1,2)$ and its minimal resolution is a minimal surface
of general type with $p_g=0$ and $K^2=3$,  then there is a
$C_3$-cover $X\to Z$ branched at the three singular points of $Z$
such that $X$ is a fake projective plane.
\item If $Z$ is a ${\mathbb Q}$-homology projective plane with $4$
singular points of type $\frac{1}{3}(1,2)$ and its minimal
resolution is a minimal surface of general type with $p_g=0$ and
$K^2=1$, then there is a $C_3$-cover $Y\to Z$ branched at three of
the four singular points of $Z$ and a $C_3$-cover $X\to Y$
branched at the three singular points on $Y$, the pre-image of the
remaining singularity on $Z$, such that $X$ is a fake projective
plane.
\item If $Z$ is a ${\mathbb Q}$-homology
projective plane with $3$ singular points of type
$\frac{1}{7}(1,5)$ and its minimal resolution is a $(2,3)$- or
$(2,4)$-elliptic surface, then there is a $C_7$-cover $X\to Z$
branched at the three singular points of $Z$ such that $X$ is a
fake projective plane.
 \item If $Z$ is a ${\mathbb Q}$-homology projective plane with $4$
singular points, 3 of type $\frac{1}{3}(1,2)$ and one of type
$\frac{1}{7}(1,5)$, and its minimal resolution is a $(2,3)$- or
$(2,4)$-elliptic surface, then there is a $C_3$-cover $Y\to Z$
 branched at the three singular points of type $\frac{1}{3}(1,2)$ and a $C_7$-cover $X\to
Y$ branched at the three singular points, the pre-image of the
singularity on $Z$ of type $\frac{1}{7}(1,5)$, such that $X$ is a
fake projective plane.
\end{enumerate}
\end{theorem}

In the case $(4)$, we give a detailed information on the types of
singular fibres of the elliptic fibration on the minimal
resolution of $Z$.

\begin{theorem}\label{main2}
Let $Z$ be a ${\mathbb Q}$-homology projective plane with $4$
singular points, 3 of type $\frac{1}{3}(1,2)$ and one of type
$\frac{1}{7}(1,5)$. Assume that its minimal resolution $\tilde{Z}$
is a $(2,3)$-elliptic surface.
 Then the following hold true.
\begin{enumerate}
\item The triple cover $Y$ of $Z$ branched at the three singular
points of type $\frac{1}{3}(1,2)$ is a ${\mathbb Q}$-homology
projective plane with $3$ singular points of type
$\frac{1}{7}(1,5)$. The minimal resolution $\tilde{Y}$ of $Y$ is a
$(2,3)$-elliptic surface, and every fibre of the elliptic
fibration on $\tilde{Z}$ does not split in $\tilde{Y}$. \item The
elliptic fibration on $\tilde{Z}$ has $4$ singular fibres of type
$\mu_1I_3+ \mu_2I_3+\mu_3I_3+\mu_4I_3$, where $\mu_i$ is the
multiplicity of the fibre.
\item The elliptic fibration on $\tilde{Y}$ has $4$ singular
fibres of type $\mu I_9+ \mu_1I_1+\mu_2I_1+ \mu_3I_1$.
\end{enumerate}
\end{theorem}

The case where $\tilde{Z}$ is a $(2,4)$-elliptic surface was
treated in \cite{K10}. The assertions (2) and (3) of Theorem
\ref{main2} were given without proof in Corollary 4.12 and 1.4 of
\cite{K08}.

As a consequence of Theorem \ref{main1} and the result of
Cartwright and Steger (\cite{CS}, \cite{CS2}), we give a
classification of ${\mathbb Q}$-homology projective planes with
cusps, i.e., singularities of type $\frac{1}{3}(1,2)$, only.

\begin{theorem}\label{main3}
Let $Z$ be a ${\mathbb Q}$-homology projective plane with cusps
only.
 Then $Z$ is isomorphic to one of the following:
\begin{enumerate}
\item $X/C_3$, where $X$ is a fake projective plane with an order 3
automorphism;
\item $X/C_3^2$, where $X$ is a fake projective plane
with $Aut(X)=C_3^2$;
\item $\mathbb{P}^2/\langle\sigma\rangle$, where $\sigma$ is the order 3 automorphism given by
$$\sigma(x,y,z)=(x,\omega y,\omega^2 z);$$
\item $\mathbb{P}^2/\langle\sigma, \tau\rangle$, where $\sigma$
and $\tau$ are the commuting order 3 automorphisms given by
$$\sigma(x,y,z)=(x,\omega y,\omega^2 z), \quad \tau(x,y,z)=(z, ax,
a^{-1}y),$$ where $a$ is a non-zero constant and
$\omega=exp(\frac{2\pi\sqrt{-1}}{3}).$
\end{enumerate}
\end{theorem}

\begin{remark}  In differential topology, they use two notions ``exotic $\mathbb{P}^2$" and ``fake $\mathbb{P}^2$".
An exotic $\mathbb{P}^2$ is a simply connected symplectic
  4-manifold homeomorphic to, but
  not diffeomorphic to $\mathbb{P}^2$. The existence of such a
 4-manifold is not known yet. It does not exist in complex
 category.
\end{remark}


\bigskip
{\bf Notation}
\begin{itemize}
\item $K_Y$: the canonical class of $Y$ \item $b_i(Y):$ the $i$-th
Betti number of $Y$ \item $e(Y):$ the topological Euler number of
$Y$  \item $q(X):=\dim H^1(X,\mathcal{O}_{X})$, the irregularity
of a surface $X$ \item $p_g(X):=\dim H^2(X,\mathcal{O}_{X})$, the
geometric genus of a surface $X$
\end{itemize}

\section{Preliminaries}

First, we recall the toplogical and holomorphic Lefschetz fixed
point formulas.

\medskip \noindent {\bf Toplogical Lefschetz Fixed Point Formula.}
{\it Let $M$ be a topological manifold of dimension $m$ admitting
a homeomrphism $\sigma$.  Then the Euler number of the fixed locus
$M^{\sigma}$ of $\sigma$ is equal to the alternating sum of the
trace of $\sigma^*$ acting on $H^j(M, \mathbb Z)$, i.e.,}
$$e(M^{\sigma})=\sum_{j=0}^m (-1)^j Tr\sigma^* |H^j(M, \mathbb
Z).$$

\medskip\noindent {\bf Holomorphic Lefschetz Fixed Point
Formula.}(\cite{AS3}, p. 567) {\it Let $M$ be a complex manifold
of dimension $2$ admitting an automorphism $\sigma$. Let $p_1,
\ldots, p_l$ be the isolated fixed points of $\sigma$ and
$R_1,\ldots, R_k$ be the $1$-dimensional components of the fixed
locus $S^{\sigma}$. Then
$$\sum_{j=0}^2 (-1)^j Tr\sigma^* |H^j(M, {\mathcal O}_M)=\sum_{j=1}^l\dfrac{1}{\det(I-d\sigma)|T_{p_j}}+
\sum_{j=1}^k\Big\{\dfrac{1-g(R_j)}{1-\xi_j}-\dfrac{\xi_j
R_j^2}{(1-\xi_j)^2}\Big\},$$ where $T_{p_j}$ is the tangent space
at $p_j$, $g(R_j)$ is the genus of $R_j$ and $\xi_j$ is
the eigenvalue of the differential $d\sigma$ acting on the normal bundle of $R_j$ in $M$.\\
Assume further that $\sigma$ is of finite and prime order $p$.
Then
$$\dfrac{1}{p-1}\sum_{i=1}^{p-1}\sum_{j=0}^2 (-1)^j Tr\sigma^{i*} |H^j(M, {\mathcal O}_M)=
\sum_{i=1}^{p-1}a_ir_i+
\sum_{j=1}^k\Big\{\frac{1-g(R_j)}{2}+\frac{(p+1)R_j^2}{12}
\Big\},$$ where $r_i$ is the number of isolated fixed points of
$\sigma$ of type $\frac{1}{p}(1,i)$, and
$$a_i=\frac{1}{p-1}\sum_{j=1}^{p-1}\frac{1}{(1-\zeta^j)(1-\zeta^{ij})}$$
with $\zeta=\exp(\frac{2\pi\sqrt{-1}}{p})$, e.g.,
$a_1=\frac{5-p}{12},\,\,a_2=\frac{11-p}{24}$, etc.}

\bigskip
For a complex manifold $M$ of dimension $2$ with
$K_M^2=3c_2(M)=9$, it is known that $$p_g(M)=q(M)\le 2.$$ Indeed,
such a surface $M$ has $\chi(\mathcal O_M)=1$, $p_g(M)=q(M)$, and
is a ball-quotient or $\mathbb P^2$. Since $c_2(M)=3$, $M$ cannot
be fibred over a curve of genus $\ge 2$. Thus by Castelnuovo-de
Franchis theorem, $p_g(M)\ge 2q(M)-3$, which implies
$p_g(M)=q(M)\le 3$. The case of $p_g(M)=q(M)= 3$ was eliminated by
the classification result of Hacon and Pardini \cite{HP} (see also
\cite{Pirola} and \cite{CCM}).

\begin{proposition}\label{lf} Let $M$ be a complex manifold $M$ of dimension $2$ with $K_M^2=3c_2(M)=9$. Then, the following hold true.
\begin{enumerate}
        \item If $M$ admits an order 7 automorphism $\sigma$ with isolated
        fixed points only, then $p_g(M/\langle\sigma\rangle)=q(M/\langle\sigma\rangle)=p_g(M)=q(M)$,
        and $M/\langle\sigma\rangle$ has either 3 singular points of type $\frac{1}{7}(1,5)$ or 2 singular points of type $\frac{1}{7}(1,2)$
        and 1 singular point of type $\frac{1}{7}(1,6)$.
        \item If $M$ has $p_g(M)=q(M)=1$ and admits an order 3 automorphism $\sigma$ with isolated
        fixed points only, then
        \begin{enumerate}
        \item $p_g(M/\langle\sigma\rangle)=q(M/\langle\sigma\rangle)=0$,
        and $M/\langle\sigma\rangle$ has 6 singular points of type $\frac{1}{3}(1,1)$; or
        \item $p_g(M/\langle\sigma\rangle)=1, \,\, q(M/\langle\sigma\rangle)=0$,
        and $M/\langle\sigma\rangle$ has 3 singular points of type $\frac{1}{3}(1,1)$ and 6 singular points of type $\frac{1}{3}(1,2)$; or
        \item $p_g(M/\langle\sigma\rangle)=q(M/\langle\sigma\rangle)=1$,
        and $M/\langle\sigma\rangle$ has 3 singular points of type $\frac{1}{3}(1,2)$.
\end{enumerate}
\end{enumerate}
\end{proposition}

\begin{proof}
Note that $M$ cannot admit an automorphism of finite order acting
freely, because $\chi(\mathcal O_M)=1$ not divisible by any
integer $\ge 2$.

\medskip\noindent
(1) By Hodge decomposition theorem,
$$Tr\sigma^* |H^1(M, \mathbb Z)=Tr\sigma^* |H^1(M, \mathbb
C)=Tr\sigma^{*} |(H^{0,1}(M)\oplus H^{1,0}(M)).$$ Note that this
number is an integer. Let $\zeta=\exp(\frac{2\pi\sqrt{-1}}{7})$.

Assume that $p_g(M)=q(M)=2$. Let $\zeta^i$ and $\zeta^j$ be the
eigenvalues of $\sigma^{*}$ acting on $H^{0,1}(M)$. Then
$$Tr\sigma^* |H^1(M, \mathbb
Z)=\zeta^i+\zeta^j+\bar{\zeta}^i+\bar{\zeta}^j,$$ and this is an
integer iff  $\zeta^i=\zeta^j=1$. This implies that $ Tr\sigma^*
|H^{0,1}(M)=2$ and $q(M/\langle\sigma\rangle)=q(M)=2$. By the
Toplogical Lefschetz Fixed Point Formula, $e(M^{\sigma})=-6+
Tr\sigma^* |H^2(M, \mathbb Z)$, so $6< Tr\sigma^* |H^2(M, \mathbb
Z)$. Since $$\rank H^2(M, \mathbb Z)=1+4q(M)=9,$$ it follows that
$Tr\sigma^* |H^2(M, \mathbb Z)=9$ and $e(M^{\sigma})=3$. In
particular, $Tr\sigma^* |H^{0,2}(M)=2$ and
$p_g(M/\langle\sigma\rangle)=p_g(M)=2$. By the Holomorphic
Lefschetz Fixed Point Formula,
$$1=-\dfrac{1}{6}r_1+\dfrac{1}{6}(r_2+r_4)+\dfrac{1}{3}(r_3+r_5)+ \dfrac{2}{3}r_6,$$
where $r_i$ is the number of isolated fixed points of $\sigma$ of
type $\frac{1}{7}(1,i)$. Since $$\sum r_i=e(M^{\sigma})=3,$$ we
have two solutions: $$r_3+r_5=3, r_1=r_2=r_4=r_6=0;\,\,\,
r_2+r_4=2, r_6=1, r_1=r_3=r_5=0.$$

Assume that $p_g(M)=q(M)=1$. By the same argument, $Tr\sigma^*
|H^{0,1}(M)=1$, $Tr\sigma^* |H^2(M, \mathbb Z)=5$,
$e(M^{\sigma})=3$ and  $Tr\sigma^* |H^{0,2}(M)=1$.

Assume that $p_g(M)=q(M)=0$. Then $Tr\sigma^*
|H^{0,1}(M)=Tr\sigma^* |H^{0,2}(M)=0$, $Tr\sigma^* |H^2(M, \mathbb
Z)=1$ and $e(M^{\sigma})=3$.

\medskip\noindent
(2) First note that $p_g(M/\langle\sigma\rangle)\le 1$ and
$q(M/\langle\sigma\rangle)\le 1$.\\ Let $\zeta^i$ and $\zeta^j$ be
the eigenvalues of $\sigma^{*}$ acting on $H^{0,1}(M)$ and
$H^{0,2}(M)$, respectively, where
$\zeta=\exp(\frac{2\pi\sqrt{-1}}{3})$.\\
Also note that $\rank H^{1,1}(M)=1+2q(M)=3$. Since $\sigma^{*}$
fixes the class of a fibre of the Albanese fibration $X\to Alb(X)$
and the class of $K_X$, we have $Tr\sigma^{*}
|H^{1,1}(M)=2+\zeta^k$.

Assume that
$p_g(M/\langle\sigma\rangle)=q(M/\langle\sigma\rangle)=0$.  Then
$\zeta^i\neq 1$ and $\zeta^j\neq 1$, hence
$$Tr\sigma^* |H^1(M, \mathbb Z)=Tr\sigma^{*} |(H^{0,1}(M)\oplus
H^{1,0}(M))=\zeta^i+\bar{\zeta}^i=-1,$$ $$Tr\sigma^{*}
|(H^{0,2}(M)\oplus H^{2,0}(M))=\zeta^j+\bar{\zeta}^j=-1.$$ The
latter implies that $Tr\sigma^{*} |H^{1,1}(M)$ is an integer,
hence $Tr\sigma^{*} |H^{1,1}(M)=3$. Then
 by the Toplogical
Lefschetz Fixed Point Formula, $e(M^{\sigma})=6$. By the
Holomorphic Lefschetz Fixed Point Formula,
$$1=\dfrac{1}{6}r_1+\dfrac{1}{3}r_2,$$
where $r_i$ is the number of isolated fixed points of $\sigma$ of
type $\frac{1}{3}(1,i)$. Since $r_1+r_2=e(M^{\sigma})=6$, we have
a unique solution: $r_1=6,\, r_2=0$. This gives (a).

Assume that $p_g(M/\langle\sigma\rangle)=1$ and
$q(M/\langle\sigma\rangle)=0$.  Then $\zeta^i\neq 1$ and $\zeta^j=
1$, hence
$$Tr\sigma^* |H^1(M, \mathbb Z)=Tr\sigma^{*} |(H^{0,1}(M)\oplus
H^{1,0}(M))=\zeta^i+\bar{\zeta}^i=-1,$$ $$Tr\sigma^{*}
|(H^{0,2}(M)\oplus H^{2,0}(M))=1+1=2.$$ The latter implies that
$Tr\sigma^{*} |H^{1,1}(M)$ is an integer, hence $Tr\sigma^{*}
|H^{1,1}(M)=3$. Then by the Toplogical Lefschetz Fixed Point
Formula, $e(M^{\sigma})=9$. By the Holomorphic Lefschetz Fixed
Point Formula,
$$\dfrac{1}{2}\{(1-\zeta^i+1)+(1-\zeta^{2i}+1)\}=\dfrac{5}{2}=\dfrac{1}{6}r_1+\dfrac{1}{3}r_2.$$
 Since $r_1+r_2=9$, we have
a unique solution: $r_1=3,\, r_2=6$. This gives (b).

Assume that
$p_g(M/\langle\sigma\rangle)=q(M/\langle\sigma\rangle)=1$.  Then
Then $\zeta^i=\zeta^j= 1$, hence
$$Tr\sigma^{*} |(H^{0,1}(M)\oplus
H^{1,0}(M))=Tr\sigma^{*} |(H^{0,2}(M)\oplus H^{2,0}(M))=2,$$
$Tr\sigma^{*} |H^{1,1}(M)=3$ and  $e(M^{\sigma})=3$. By the
Holomorphic Lefschetz Fixed Point Formula,
$$1=\dfrac{1}{6}r_1+\dfrac{1}{3}r_2.$$
Since $r_1+r_2=3$, we have a unique solution: $r_1=0,\, r_2=3$.
This gives (c).

Assume that $p_g(M/\langle\sigma\rangle)=0$ and
$q(M/\langle\sigma\rangle)=1$.  Then $\zeta^i=1$ and $\zeta^j\neq
1$, hence $$Tr\sigma^{*} |(H^{0,1}(M)\oplus H^{1,0}(M))=2,\,\,\,
Tr\sigma^{*} |(H^{0,2}(M)\oplus
H^{2,0}(M))=\zeta^j+\bar{\zeta}^j=-1,$$ $Tr\sigma^{*}
|H^{1,1}(M)=3$ and $e(M^{\sigma})=0$. Thus $\sigma$ acts freely, a
contradiction.
\end{proof}

\begin{proposition}\label{ab} Let $M$ be an abelian surface. Assume that it admits an order 3 automorphism
$\sigma$ such that $p_g(M/\langle\sigma\rangle)=0$. Then
$b_2(M/\langle\sigma\rangle)=4$ or $2$.
\end{proposition}

\begin{proof}  First note that $p_g(M)= 1$ and $\rank H^{1,1}(M)=4$. Let
$\zeta=\exp(\frac{2\pi\sqrt{-1}}{3})$.\\ Let $\zeta^k$ be the
eigenvalue of $\sigma^{*}$ acting on $H^{0,2}(M)$. Since
$p_g(M/\langle\sigma\rangle)=0$, we have $\zeta^k\neq 1$, hence
$$Tr\sigma^{*}
|(H^{0,2}(M)\oplus H^{2,0}(M))=\zeta^k+\bar{\zeta}^k=-1.$$ It
implies that $Tr\sigma^{*} |H^{1,1}(M)$ is an integer, hence is
equal to 4, 1 or $-2$.  The last possibility can be ruled out, as
there is a $\sigma$-invariant ample divisor yielding a
$\sigma^*$-invariant vector in $H^{1,1}(M)$. Finally note that
$b_2(M/\langle\sigma\rangle)=\rank H^{1,1}(M)^{\sigma}$.
\end{proof}

\begin{remark} If in addition, $q(M/\langle\sigma\rangle)=0$, then
either
\begin{enumerate} \item $r_2=0$,\,\, $r_1-\sum R_j^2=9$,\,\,
$b_2(M/\langle\sigma\rangle)=4$; or
\item $r_2=3$,\,\, $r_1-\sum R_j^2=3$,\,\,
$b_2(M/\langle\sigma\rangle)=2$.
\end{enumerate}
 Here $r_i$ is the number of isolated fixed points of type
 $\frac{1}{3}(1,i)$, and $\cup R_j$ is the 1-dimensional fixed locus of
 $\sigma$.
\end{remark}

\begin{proposition}\label{p=q=2} Let $M$ be a surface of general type with $p_g(M)=q(M)=2$.
Assume that it admits an order 3 automorphism $\sigma$ with
isolated fixed points only such that
$p_g(M/\langle\sigma\rangle)=q(M/\langle\sigma\rangle)=0$. Let
$\bar{a}: M/\langle\sigma\rangle\to Alb(M)/\langle\sigma\rangle$
be the map induced by the Albanese map $a: M\to Alb(M)$. Then
$\bar{a}$ cannot factor through a surjective map
$M/\langle\sigma\rangle\to N$ to a normal projective surface $N$
with Picard number 1.
\end{proposition}

\begin{proof}
Suppose that $\bar{a}$ factors through a surjective map
$M/\langle\sigma\rangle\to N$ to a normal projective surface $N$
with Picard number 1, i.e.,
$$\bar{a}: M/\langle\sigma\rangle\to N \to Alb(M)/\langle\sigma\rangle.$$
Let $b: N \to Alb(M)/\langle\sigma\rangle$ be the second map.
Since a normal projective surface with Picard number 1 cannot be
fibred over any curve, the map $b$ is surjective. Since
$p_g(M/\langle\sigma\rangle)=q(M/\langle\sigma\rangle)=0$, we have
$$p_g(N)=q(N)=0\,\,\,\textrm{and}\,\,\, p_g(Alb(M)/\langle\sigma\rangle)=q(Alb(M)/\langle\sigma\rangle)=0.$$
Since $Alb(M)/\langle\sigma\rangle$ has quotient singularities
only, its minimal resolution has $p_g=q=0$, hence
$$\Pic(Alb(M)/\langle\sigma\rangle)\otimes \mathbb Q\cong
H^2(Alb(M)/\langle\sigma\rangle, \mathbb Q).$$ By Proposition
\ref{ab}, $Alb(M)/\langle\sigma\rangle$ has Picard number 4 or 2.
This is a contradiction, as a normal projective surface with
Picard number 1 cannot be mapped surjectively onto a surface with
Picard number $\ge 2$.
\end{proof}

 Let $S$ be a normal projective surface with quotient singularities and $$f : S'
\rightarrow S$$ be a minimal resolution of $S$. It is well-known
that quotient singularities are log-terminal
 singularities. Thus one can write the adjunction formula, $$K_{S'} \underset{num}{\equiv} f^{*}K_S -
 \sum_{p \in Sing(S)}{\mathcal{D}_p},$$ where $\mathcal{D}_p = \sum(a_jA_j)$ is an effective
 $\mathbb{Q}$-divisor with $0 \leq a_j < 1$  supported on $f^{-1}(p)=\cup A_j$   for each singular point $p$.
It implies that
\[K^2_S = K^2_{S'} - \sum_{p}{\mathcal{D}_p^2}= K^2_{S'} +\sum_{p}{\mathcal{D}_pK_{S'}}.
\]
The coefficients of the $\mathbb{Q}$-divisor $\mathcal{D}_p$ can
be obtained by solving the equations
$$\mathcal{D}_pA_j=-K_{S'}A_j=2+A_j^2$$  given by the adjunction formula for each exceptional curve
$A_j\subset f^{-1}(p)$.

\section{The Proof of Theorem \ref{main1}}

\subsection{The case: $Z$ has 3 singular points of type
$\frac{1}{3}(1,2)$}

Let $p_1, p_2, p_3$ be the three singular points  of $Z$ of type
$\frac{1}{3}(1,2)$, and $\tilde{Z}\to Z$ be the minimal
resolution.

\begin{lemma}\label{3A2} There is a $C_3$-cover $X\to Z$ branched at the
three singular points of $Z$.
\end{lemma}

\begin{proof}
We use a lattice theoretic argument. Consider the cohomology
lattice $$H^2(\tilde{Z}, \mathbb{Z})_{free}:=H^2(\tilde{Z},
\bbZ)/(torsion)$$ which is unimodular of signature $(1, 6)$ under
intersection pairing. Since $Z$ is a $\mathbb Q$-homology
projective plane, $p_g(\tilde{Z})=q(\tilde{Z})=0$ and hence
Pic$(\tilde{Z})=H^2(\tilde{Z}, \bbZ)$. Let $\mathcal R_i\subset
H^2(\tilde{Z}, \mathbb{Z})_{free}$ be the sublattice spanned by
the numerical classes of the components of $f^{-1}(p_i)$. Consider
the sublattice $\mathcal R_1\oplus\mathcal R_2\oplus\mathcal R_3$.
Its discriminant group is 3-elementary of length 3, and its
orthogonal complement is of rank 1. It follows that there is a
divisor class $L \in$ Pic$(\tilde{Z})$ such that $$3L=B+\tau$$ for
some torsion divisor $\tau$,  where $B$ is an integral divisor
supported on the six $(-2)$-curves contracted to the points $p_1,
p_2, p_3$ by the map $\tilde{Z}\to Z$. Here all coefficients of
$B$ are greater than 0 and less than 3.

If $\tau=0$, $L$ gives a $C_3$-cover of $\tilde{Z}$ branched along
$B$, hence yielding a $C_3$-cover $X\to Z$ branched at the three
points $p_1, p_2, p_3$. Clearly, $X$ is a nonsingular surface.

 If $\tau\neq 0$,
let $m$ denote the order of $\tau$. Write $m=3^tm'$ with $m'$ not
divisible by 3. By considering $3(m'L)=m'B+ m'\tau$, and by
putting $B'=m'B$(modulo 3), $\tau'=m'\tau$,  we may assume that
$\tau$ has order $3^t$. The torsion bundle $\tau$ gives an
unramified $C_{3^t}$-cover $$p: V\to \tilde{Z}.$$ Let $g$ be the
corresponding automorphism of $V$. Pulling $3L=B+\tau$ back to
$V$, we have
$$3p^*L=p^*B.$$ Obviously, $g$ can be linearized on the line bundle $p^*L$,
hence gives an automorphism of order $3^t$ of the total space of
$p^*L$. Let $V'\to V$ be the $C_3$-cover given by $p^*L$. We
regard $V'$ as a subvariety of the total space of $p^*L$. Since
$g$ leaves invariant the set of local defining equations for $V'$,
$g$ restricts to an automorphism of $V'$ of order $3^t$. Thus we
have a $C_3$-cover
$$V'/\langle g\rangle \to \tilde{Z}.$$ This yields a $C_3$-cover
$X\to Z$ branched at the three points $p_1, p_2, p_3$. Clearly,
$X$ is a nonsingular surface.
\end{proof}

Since $Z$ has only rational double points, the adjunction formula
gives $K_Z^2=K_{\tilde{Z}}^2=3$. Hence $K_X^2=3K_{Z}^2=9$.  The
smooth part $Z^0$ of $Z$ has Euler number
$e(Z^0)=e(\tilde{Z})-9=0$, so $e(X)=3e(Z^0)+3=3$. This shows that
$X$ is a ball quotient with $p_g(X)=q(X)$. It is known that such a
surface has $p_g(X)=q(X)\le 2$. (See the paragraph before
Proposition \ref{lf}.) In our situation $X$ admits an order 3
automorphism, and Proposition \ref{lf} eliminates the possibility
of $p_g(X)=q(X)= 1$.

It remains to exclude the possibility of $p_g(X)=q(X)= 2$.\\
Suppose that $p_g(X)=q(X)= 2$. Consider the Albanese map $a: X\to
Alb(X)$. It induces a map $\bar{a}: Z\to Alb(X)/\sigma$, where
$\sigma$ is the order 3 automorphism of $X$ corresponding to the
$C_3$-cover $X\to Z$. Since $Z$ has Picard number 1 and
$p_g(Z)=q(Z)=0$,  Proposition \ref{p=q=2} gives a contradiction.

\subsection{The case: $Z$ has 4 singular points of type
$\frac{1}{3}(1,2)$}

Let $p_1, p_2, p_3, p_4$ be the four singular points  of $Z$, and
$f:\tilde{Z}\to Z$  the minimal resolution.

\begin{lemma}\label{4A2} If there is a $C_3$-cover $Y\to Z$ branched at three of
the four singular points of $Z$, then the minimal resolution
$\tilde{Y}$ of $Y$ has $K_{\tilde{Y}}^2=3$, $e(\tilde{Y})=9$ and
$p_g(\tilde{Y})=q(\tilde{Y})=0$.
\end{lemma}

\begin{proof} We may assume that the three points are $p_1, p_2,
p_3$. Note that $Y$ has 3 singular points of type
$\frac{1}{3}(1,2)$, the pre-image of $p_4$. Let $\tilde{Y}\to Y$
be the minimal resolution. It is easy to see that
$K_{\tilde{Y}}^2=3$, $e(\tilde{Y})=9$ and
$p_g(\tilde{Y})=q(\tilde{Y})$.

Suppose that $p_g(\tilde{Y})=q(\tilde{Y})=1$. Consider the
Albanese fibration $\tilde{Y}\to Alb(\tilde{Y})$. It induces a
fibration $Y\to Alb(\tilde{Y})$. Let $\sigma$ be the order 3
automorphism of $Y$ corresponding to the $C_3$-cover $Y\to Z$. It
induces a fibration $\phi:\tilde{Z}\to
Alb(\tilde{Y})/\langle\sigma\rangle$. Since $q(Z)=0$, we have
$Alb(\tilde{Y})/\langle\sigma\rangle\cong\mathbb P^1$. The eight
$(-2)$-curves of $\tilde{Z}$ are contained in a union of fibres of
$\phi$. It follows that $\tilde{Z}$ has Picard number $\ge
8+2=10$, a contradiction.

Suppose that $p_g(\tilde{Y})=q(\tilde{Y})=2$. Consider the
Albanese map $a: \tilde{Y}\to Alb(\tilde{Y})$.  It contracts the
six $(-2)$-curves of $\tilde{Y}$, hence the induced map $\bar{a}:
\tilde{Y}/\langle\sigma\rangle\to
Alb(\tilde{Y})/\langle\sigma\rangle$ factors through a surjective
map $\tilde{Y}/\langle\sigma\rangle\to Z$, where $\sigma$ is the
order 3 automorphism of $\tilde{Y}$ corresponding to the
$C_3$-cover $Y\to Z$. Since $Z$ has Picard number 1 and
$p_g(Z)=q(Z)=0$, Proposition \ref{p=q=2} gives a contradiction.

The possibility of $p_g(\tilde{Y})=q(\tilde{Y})\ge 3$ can be ruled
out by considering a $C_3$-cover $X\to Y$ branched at the three
singular points of $Y$. See the paragraph below Lemma
\ref{2covers}.
\end{proof}

\begin{lemma}\label{2covers} There is a $C_3$-cover $Y\to Z$ branched at three of
the four singular points of $Z$, and a $C_3$-cover $X\to Y$
branched at the three singular points of $Y$.
\end{lemma}

\begin{proof}
The existence of two $C_3$-covers can be proved by a lattice
theoretic argument. Note that Pic$(\tilde{Z})=H^2(\tilde{Z},
\bbZ)$.  We know that $H^2(\tilde{Z}, \bbZ)_{free}$ is a
unimodular lattice of signature $(1, 8)$ under intersection
pairing. Let $\mathcal R_i\subset H^2(\tilde{Z},
\mathbb{Z})_{free}$ be the sublattice spanned by the numerical
classes of the components of $f^{-1}(p_i)$. Consider the
sublattice $\mathcal R_1\oplus\mathcal R_2\oplus\mathcal
R_3\oplus\mathcal R_4$. Its discriminant group is 3-elementary of
length 4, and its orthogonal complement is of rank 1. It follows
that there are two divisor classes $L_1, L_2 \in$ Pic$(\tilde{Z})$
such that $$3L_1=B_1+\tau_1, \quad 3L_2=B_2+\tau_2$$ for some
torsion divisors $\tau_i$, where $B_i$ is an integral divisor
supported on the six $(-2)$-curves lying over three of the four
points $p_1, p_2, p_3, p_4$. We may assume that $B_i$ is supported
on $ \cup_{j\neq i} f^{-1}(p_j)$ and all coefficients of $B_i$ are
greater than 0 and less than 3.

By the same argument as in Lemma \ref{3A2}, we can take a
$C_3$-cover $Y\to Z$ branched at the three points $p_2, p_3, p_4$.
Then $Y$ has 3 singular points of type $\frac{1}{3}(1,2)$, the
pre-image of $p_1$. This can be done by using the line bundle
$L_1$ if $\tau_1=0$. Otherwise, we first take an unramified cover
$p: V\to \tilde{Z}$ corresponding to $\tau_1$ and then lift the
covering automorphism $g$ to the $C_3$-cover $V'\to V$ given by
$p^*L_1$, then take the quotient $V'/\langle g\rangle$.

 Let $\psi:\bar{Y}\to\tilde{Z}$ be the
$C_3$-cover corresponding to the $C_3$-cover $Y\to Z$, composed
with a normalization. Then $\bar{Y}$ is a normal surface and there
is a surjection $f:\bar{Y}\to\tilde{Y}$. Now
$$3f_*(\psi^*L_2)=f_*(\psi^*B_2)+f_*(\psi^*\tau_2)$$ and
$f_*(\psi^*B_2)$ is an integral divisor supported on the
exceptional locus of $\tilde{Y}\to Y$ with coefficients greater
than 0 and less than 3. Now by the same argument as in Lemma
\ref{3A2}, there is a $C_3$-cover $X\to Y$ with $X$ nonsingular.
\end{proof}

It is easy to see that $K_{X}^2=9$, $e(X)=3$ and $p_g(X)=q(X)$.
Such a surface has $p_g(X)=q(X)\le 2$. (See the paragraph before
Proposition \ref{lf}.) It implies that $p_g(Y)=q(Y)\le 2$, which
completes the proof of Lemma \ref{4A2}.

By Lemma \ref{4A2}, $p_g(Y)=q(Y)=0$, so $Y$ has Picard number 1
and has three singular points of type $\frac{1}{3}(1,2)$. Then by
the previous subsection, $p_g(X)=q(X)=0$.

\subsection{The case: $Z$ has 3 singular points of type
$\frac{1}{7}(1,5)$}

Let $p_1, p_2, p_3$ be the three singular points  of $Z$ of type
$\frac{1}{7}(1,5)$. Then there is a $C_7$-cover $X\to Z$ branched
at the three points. In the case of $\pi_1(Z)=\{1\}$, this was
proved in \cite{K06}, p922. In our general situation, we consider
the lattice Pic$(\tilde{Z})/$(torsion), where $\tilde{Z}\to Z$ is
the minimal resolution. Then by the same lattice theoretic
argument as in \cite{K06}, there is a divisor class $L \in$
Pic$(\tilde{Z})=H^2(\tilde{Z}, \bbZ)$ such that $7L=B+\tau$ for
some torsion divisor $\tau$,  where $B$ is an integral divisor
supported on the exceptional curves of the map $\tilde{Z}\to Z$.
Here all coefficients of $B$ are not equal to 0 modulo 7. If
$\tilde{Z}$ is a $(2,4)$-elliptic surface and if $\tau\neq 0$,
then $2\tau=0$. By considering $7(2L)=2B$, and by putting $L'=2L$
and $B'=2B$, we get $7L'=B'$. This implies the existence of a
$C_7$-cover $X\to Z$ branched at the three points $p_1, p_2, p_3$.
Then $X$ is a nonsingular surface.

Note that $K_{\tilde{Z}}^2=0$. So by the adjunction formula,
$K_{Z}^2=\frac{9}{7}$. It is easy to see that $K_{X}^2=9$,
$e(X)=3$ and $p_g(X)=q(X)$. Such a surface has $p_g(X)=q(X)\le 2$.
(See the paragraph before Proposition \ref{lf}.) Now by
Proposition \ref{lf}, $p_g(X)=q(X)= 0$.

\subsection{The case: $Z$ has 3 singular points of type
$\frac{1}{3}(1,2)$ and one of type $\frac{1}{7}(1,5)$}

Let $\tilde{Z}\to Z$ be the minimal resolution. Then $\tilde{Z}$
is a  $(2,3)$- or $(2,4)$-elliptic surface. Let
$$\phi:\tilde{Z}\to\mathbb P^1$$ be the elliptic fibration. Let
$Z'\to Z$ be the minimal resolution of the singular point of type
$\frac{1}{7}(1,5)$. Then $\phi:\tilde{Z}\to\mathbb P^1$ induces an
elliptic fibration
$$\phi':Z'\to\mathbb P^1.$$

\begin{lemma}\label{21to7}  \begin{enumerate}
\item There is a $C_3$-cover $Y\to Z$ branched at the three points of type
$\frac{1}{3}(1,2)$. The cover $Y$ has 3 singular points of type
$\frac{1}{7}(1,5)$.
\item The minimal resolution $\tilde{Y}$ of $Y$ is a
$(2,3)$- or $(2,4)$-elliptic surface. Its multiplicities are the
same as those of  $\tilde{Z}$.  Furthermore, every fibre of
$\tilde{Z}$ does not split in $\tilde{Y}$.
 \end{enumerate}
\end{lemma}

\begin{proof}  We may assume that $\tilde{Z}$ is a  $(2,3)$-elliptic
surface.  The case of $(2,4)$-elliptic surfaces was proved in
\cite{K10}.

(1) The existence of the triple cover can be proved in the same
way as in \cite{K06}, p920-921. Note that $Y$ has 3 singular
points of type $\frac{1}{7}(1,5)$, the pre-image of the singular
point of $Z$ of type $\frac{1}{7}(1,5)$.

(2) Consider the $C_3$-cover $\tilde{Y}\to Z'$ branched at the
three singular points of $Z'$. The elliptic fibration
$\phi':Z'\to\mathbb P^1$ induces an elliptic fibration
$\psi:\tilde{Y}\to\mathbb P^1$. Denote by $E$ the $(-3)$-curve in
$Z'$ lying over the singularity of type $\frac{1}{7}(1,5)$. It
does not pass through any of the 3 singular points of $Z'$, hence
splits in $\tilde{Y}$ to give three $(-3)$-curves $E_1$, $E_2$,
$E_3$.

Suppose that a general fibre of $Z'$ splits in $\tilde{Y}$. Since
$E$ is a 6-section, each $E_i$ will be a $2$-section of the
elliptic fibration $\psi:\tilde{Y}\to\mathbb P^1$. Thus, the map
from $E_i$ to the base curve $\mathbb P^1$ is of degree 2. It
implies that $\tilde{Y}$ has at most 2 multiple fibres and the
multiplicity of every multiple fibre is 2. Thus each multiple
fibre of $Z'$ does not split in $\tilde{Y}$. (Otherwise, it will
give 3 multiple fibres of the same multiplicity, a contradiction.)
Consider the base change map $\gamma: B_{\tilde{Y}}\cong \mathbb
P^1 \to B_{Z'}\cong \mathbb P^1$, which is of degree 3. It is
branched at the base points of the two multiple fibres of
$\phi':Z'\to\mathbb P^1$, so cannot have any more branch points.
The minimal resolution $\tilde{Z}$ contains nine curves whose dual
diagram is $$(-2)-(-2)\,\,\,\,
(-2)-(-2)\,\,\,\,(-2)-(-2)\,\,\,\,(-2)-(-2)-(-3).$$  Note that
every $(-2)$-curve on $\tilde{Z}$ is contained in a fiber.  The
eight $(-2)$-curves are contained in a union of fibres, only in
one of the following three cases. Here $\mu$ or $\mu_i$ is the
multiplicity of the fibre.

\medskip
$$(a)\,\,IV^*+\mu I_3,\,\,\, (b)\,\, IV^*+IV,\,\,\,(c)\,\, \mu_1 I_3+\mu_2I_3+\mu_3I_3+\mu_4I_3.$$

\medskip\noindent
In the first two cases, the $(-3)$-curve must intersects with
multiplicity 2 the central component of the $IV^*$-fibre. Thus,
the image in $Z'$ of the $IV^*$-fibre contains the 3 singular
points of $Z'$, so it does not split in $Y$. This means that the
base point of the $IV^*$-fibre is another branch point of the base
change map $\gamma$, a contradiction. In the last case, we also
get at least 3 branch points of $\gamma$, a contradiction.
Therefore, every fibre of $Z'$ does not split in $\tilde{Y}$. In
particular, the multiplicity of a fibre in $\tilde{Y}$ is the same
as that of the corresponding fibre in $\tilde{Z}$. Thus
$\tilde{Y}$ is an elliptic surface over ${\mathbb P}^1$ having 2
multiple fibres with multiplicity 2 and 3, resp. Since
$K_{\tilde{Z}}^2=0$ and $Z'$ has only rational double points, the
adjunction formula gives $K_{Z'}^2=K_{\tilde{Z}}^2=0$. Hence
$K_{\tilde{Y}}^2=3K_{Z'}^2=0$. In particular, $\tilde{Y}$ is
minimal. The smooth part $Z^0$ of $Z'$ has Euler number
$e(Z^0)=e(\tilde{Z})-9=3$, so $e(\tilde{Y})=3e(Z^0)+3=12$. This
shows that $\tilde{Y}$ is a $(2,3)$-elliptic surface.
\end{proof}

Now by the previous subsection, there is a $C_7$-cover $X\to Y$
branched at the three singular points such that $X$ is a fake
projective plane.

\section{Proof of Theorem \ref{main2}}

(1) was proved in Lemma \ref{21to7}.

\medskip
(2) As we have seen in the proof of Lemma \ref{21to7}, the eight
$(-2)$-curves on $\tilde{Z}$ are contained in a union of fibres,
only in one of the following three cases. Here $\mu$ or $\mu_i$ is
the multiplicity of the fibre.

$$(a)\,\,IV^*+\mu I_3,\,\,\, (b)\,\, IV^*+IV,\,\,\,(c)\,\, \mu_1 I_3+\mu_2I_3+\mu_3I_3+\mu_4I_3.$$

\medskip\noindent
Recall that every fibre in $\tilde{Z}$  does not split in
$\tilde{Y}$, and the $(-3)$-curve in $\tilde{Z}$ is a 6-section.
We will eliminate the first two cases. Let $Z'\to Z$ be the
minimal resolution of the singular point of type
$\frac{1}{7}(1,5)$.

\medskip
Case $(a):IV^*+\mu I_3$.  In this case, the surface $\tilde{Z}$
has a fibre of type $\mu' I_1$. Since the $(-3)$-curve in
$\tilde{Z}$ is a 6-section, it intersects with multiplicity 2 the
central component of the $IV^*$-fibre. Thus both the $\mu
I_3$-fibre and the $\mu' I_1$-fibre are disjoint from the branch
of the $C_3$-cover $\tilde{Y} \to Z'$. It is easy to see that
these two fibres will give a $\mu I_9$-fibre and a $\mu'
I_3$-fibre in $\tilde{Y}$, so $\tilde{Y}$ has Picard number $\ge
12$, a contradiction.

\medskip
Case $(b): IV^*+IV$. This case can be eliminated in a similar way
as above.  The $IV$-fibre on $\tilde{Z}$ does not contain any of
the $(-2)$-curves contracted by the map $\tilde{Z} \to Z'$. But
there is no unramified connected triple cover of a $IV$-fibre.

\medskip
(3) If the image in $Z'$ of the $\mu_i I_3$-fibre contains a
singular point of $Z'$, then it will give a $\mu_i I_1$-fibre in
$\tilde{Y}$. If it does not, then it will give a $\mu_i I_9$-fibre
in $\tilde{Y}$.

\section{${\mathbb Q}$-homology projective planes with cusps}

In this section we will prove Theorem \ref{main3}.

Let $Z$ be a ${\mathbb Q}$-homology projective plane with cusps,
i.e., singularities of type $\frac{1}{3}(1,2)$, only. Let $
\tilde{Z}\to Z$ be the minimal resolution.

Let $k$ be the number of cusps on $Z$. A ${\mathbb Q}$-homology
projective plane with quotient singularities can have at most 5
singular points, and the case with the maximum possible number of
quotient singularities was classified in \cite{HK}. According to
this classification, there is no ${\mathbb Q}$-homology projective
plane with 5 cusps. Thus we have $k\le 4$. It is easy to see that
$K_Z^2=K_{\tilde{Z}}^2=9-2k$. Since $K_Z^2>0$, $K_Z$ is not
numerically trivial. By Lemma 3.3 of \cite{HK}, the product of the
orders of local abelianized fundamental groups and $K_Z^2$ is a
positive square number. In our situation, the product is
$3^k(9-2k)$, and this number is a square only if $k=4$ or $3$.

Since $K_Z$ is not numerically trivial, either $K_Z$ or $-K_Z$ is
ample.

Assume that $K_Z$ is ample. Then $K_{\tilde{Z}}$ is nef, hence
$\tilde{Z}$ is a minimal surface of general type. By Theorem
\ref{main1}, $Z$ is the quotient of a fake projective plane by a
group of order 9 if $k=4$, by order 3 if $k=3$.

Assume that $-K_Z$ is ample. Then  $Z$ is a log del Pezzo surface
of Picard number 1 with 4 or 3 cusps. Assume that $Z$ has 3 cusps.
By a similar argument as in Section 2, there is a $C_3$-cover
$\mathbb{P}^2\to Z$ branched at the 3 cusps. It is easy to see
that the covering automorphism is a conjugate of the order 3
automorphism
$$\sigma:(x,y,z)\mapsto (x,\omega y,\omega^2 z).$$ Assume that $Z$ has 4
cusps. By a similar argument as in Section 2, there is a
$C_3^2$-cover $\mathbb{P}^2\to Z$ branched at the 4 cusps, the
composition of two $C_3$-covers. It is easy to see that the Galois
group is a conjugate of $\langle\sigma, \tau\rangle$, where
 $\sigma$ and
$\tau$ are the commuting order 3 automorphisms given by
$$\sigma(x,y,z)=(x,\omega y,\omega^2 z), \quad \tau(x,y,z)=(z, ax,
a^{-1}y),$$ where $a$ is a non-zero constant and
$\omega=exp(\frac{2\pi\sqrt{-1}}{3}).$

\begin{remark} (1) In the case (1) and (2), the fundamental group $\pi_1(Z)$ is given by the list of Cartwright and Steger. See Remark
0.3.

(2) One can construct a log del Pezzo surface of Picard number 1
with 4 or 3 cusps in many ways other than taking a global
quotient. One different way is to consider a rational elliptic
surface $V$ with 4 singular fibres of type $I_3$. Such an elliptic
surface can be constructed by blowing up $\mathbb P^2$ at the 9
base points of the Hesse pencil. Every section is a $(-1)$-curve.
Contracting a section, we get a nonsingular rational surface $W$
with eight $(-2)$-curves forming a diagram of type 4$A_2$.
Contracting these eight $(-2)$-curves, we get a log del Pezzo
surface of Picard number 1 with 4 cusps. On $W$, we contract a
string of two rational curves forming a diagram
$(-1)\textrm{---}(-2)$ to get a nonsingular rational surface with
six $(-2)$-curves forming a diagram of type 3$A_2$. Contracting
these six $(-2)$-curves, we get a log del Pezzo surface of Picard
number 1 with 3 cusps.
\end{remark}


\end{document}